\RequirePackage{fix-cm}
\documentclass[smallcondensed]{svjour3}     
\smartqed  
\usepackage{graphicx}
\usepackage{amsmath,amssymb}
\usepackage{theoremref}
\usepackage{cite}
\usepackage{url}
\usepackage{hyperref}

\newtheorem{thm}{Theorem}

\newtheorem{rmk}{Remark}
\newtheorem{lem}{Lemma}
\newtheorem{prop}{Proposition}

\newcommand{\R}{\mathbb R}

\begin{document}

\title{The equidistribution of Fourier coefficients of half integral weight modular forms on the plane}


\titlerunning{The equidistribution of Fourier}        

\author{Soufiane Mezroui}


\institute{Soufiane Mezroui \at
           LabTIC,\\
           SIC Department,\\
           ENSAT,\\
           Abdelmalek Essaadi University,\\
           Tangier, Morocco\\
           \email{mezroui.soufiane@yahoo.fr}
           }

\maketitle

\begin{abstract}
Let $f=\sum_{n=1}^{\infty}a(n)q^{n}\in S_{k+1/2}(N,\chi_{0})$ be a non-zero cuspidal Hecke eigenform of weight $k+\frac{1}{2}$ and the trivial nebentypus $\chi_{0}$ where the Fourier coefficients $a(n)$ are real. Bruinier and Kohnen conjectured that the signs of $a(n)$ are equidistributed. This conjecture was proved to be true by Inam, Wiese and Arias-de-Reyna for the subfamilies $\{a(t n^{2})\}_{n}$ where $t$ is a squarefree integer such that $a(t)\neq 0$. Let $q$ and $d$ be natural numbers such that $(d,q)=1$. In this work, we show that $\{a(t n^{2})\}_{n}$ is equidistributed over any arithmetic progression $n\equiv d\text{ mod }q$.
\end{abstract}

\section{Introduction}

Let $k\geq 2$, $4|N$ be integers, $\chi\pmod N$ a Dirichlet character, and let \\
$f=\sum_{n=1}^{\infty}a(n)q^{n}\in S_{k+1/2}(N,\chi)$ be a non-zero cuspidal Hecke eigenform of weight $k+\frac{1}{2}$. Applying the Shimura lift to $f$ for a fixed squarefree $t$ such that $a(t)\neq 0$, we get $F_{t}=\sum_{n=1}^{\infty}A_{t}(n)q^{n}\in S_{2k}(N/2,\chi^{2})$ the Hecke eigenform of weight $2k$. 

When $\chi=1$, Bruinier and Kohnen suggested in \cite{bruin} that half of the coefficients $a(n)$ are positive among all non-zero Fourier coefficients. This suggestion was formulated later explicitly as a conjecture in \cite{kohnen}. Assuming some error term for the convergence of the Sato-Tate distribution for integral weight modular forms in \cite{ilker3}, Inam and Wise showed when $F_t$ has no CM that half of the coefficients $a(t n^{2})$ are positive. They formulated this result in terms of Dedekind-Dirichlet density. They also  showed with Arias-de-Reyna in \cite{ilker2}, that $(a(t n^{2}))_{n\in\mathbb{N}}$ are equidistributed when $F_t$ has CM and the equidistribution was reformulated in both CM and not CM cases using Dedekind-Dirichlet and natural densities. Later, those results were obtained in \cite{ilker1} by removing the error term assumption.  

The present work gives an improvement of the Bruinier-Kohnen conjecture. Indeed, under the error term hypothesis that we will explain below, our main result is the following theorem.

\begin{thm}\thlabel{thm71}
Assume the setting of the introduction and suppose that $F_t$ does not have complex multiplication. Let $q$ be a natural number. Suppose that for all Dirichlet characters $\varepsilon\pmod q$ and all roots of unity $\xi$ such that $\xi\in\,Im\,\varepsilon$, there are $C_{\varepsilon,\xi}>0$ and $\alpha_{\varepsilon ,\xi}>0$ such that 
\begin{equation}
\left|\frac{\#\left\{p\leq x\text{ prime}\mid p\nmid N, \varepsilon(p)=\xi, \frac{A_{t}(p)}{2 a(t)p^{\frac{k-1}{2}}\chi(p)}\in [a,b]\right\}}{\pi(x)}-\frac{\mu([a,b])}{\# Im\,\varepsilon}\right| \leq \frac{C_{\varepsilon,\xi}}{x^{\alpha_{\varepsilon,\xi}}}. 
\end{equation}
Then for all integers $d$, $(d,q)=1$, the sets 
\begin{multline}
\left\{n\in\mathbb{N}\mid (n,N)=1, n\equiv d\text{ mod }q, \frac{a(t n^{2})}{\chi(n)}>0\right\}\text{ and }\\
\left\{n\in\mathbb{N}\mid (n,N)=1, n\equiv d\text{ mod }q,\frac{a(t n^{2})}{\chi(n)}<0\right\}
\end{multline}
have equal positive natural densities and both are half of the natural density of 
\begin{equation}\label{t1t}
\left\{n\in\mathbb{N}\mid (n,N)=1, n\equiv d\text{ mod }q, \frac{a(t n^{2})}{\chi(n)}\neq 0\right\}.
\end{equation} 
\end{thm}

We discuss here two aspects of this theorem. Consider first the case when $\chi=1$ and the coefficients $a(n)$ are real. Then for all natural numbers $q$ and $d$ such that $(d,q)=1$, we have
$$
\lim_{x\rightarrow +\infty}\frac{\#\left\{n\leq x\mid n\equiv d\text{ mod }q, a(t n^{2})\gtrless 0\right\}}{\#\left\{n\leq x\mid n\equiv d\text{ mod }q, a(t n^{2})\neq 0\right\}}=\frac{1}{2}\cdot
$$
This extends the results obtained in \cite{ilker2,ilker3}, and therefore, one can ask if the Bruinier-Kohnen conjecture remains true over arithmetic progressions. We have no numerical experiments yet to support this hypothesis.

Consider now the general case $f\in S_{k+1/2}(N,\chi)$. Let $q$ be a natural number, $\varepsilon\text{ mod }q$ a Dirichlet character and $\xi\in\,Im\,\varepsilon$. From the main theorem above and since the density of the set \eqref{t1t} is independent of $d$ by \thref{rem1} and \thref{rem2}, the sets 
\begin{multline}
\left\{n\in\mathbb{N}\mid (n,N)=1, \varepsilon (n)=\xi, \frac{a(t n^{2})}{\chi(n)}>0\right\}\text{ and }\\
\left\{n\in\mathbb{N}\mid (n,N)=1, \varepsilon (n)=\xi,\frac{a(t n^{2})}{\chi(n)}<0\right\}
\end{multline}
have equal positive natural densities and both are half of the natural density of 
$$
\left\{n\in\mathbb{N}\mid (n,N)=1, \varepsilon (n)=\xi, \frac{a(t n^{2})}{\chi(n)}\neq 0\right\}.
$$
In the particular case $q=N$ and $\varepsilon=\chi$, we deduce that when $\xi\neq \pm i$, 
the sets 
\begin{multline}
\left\{n\in\mathbb{N}\mid \chi (n)=\xi, Re\left(a(t n^{2})\right)>0\right\}\text{ and }\\
\left\{n\in\mathbb{N}\mid \chi (n)=\xi, Re\left(a(t n^{2})\right)<0\right\}
\end{multline}
have equal positive natural densities and both are half of the natural density of 
$$
\left\{n\in\mathbb{N}\mid \chi (n)=\xi, a(t n^{2})\neq 0\right\}.
$$
Geometrically, the coefficients $a(t n^{2})$ with $\chi (n)=\xi$ belong to the same line and they are equidistributed over it. When $\xi=\pm i$, we obtain a similar result and the coefficients $a(t n^{2})$ with $\chi (n)=i$ or $-i$  are equidistributed over the vertical line that passes through $i$ and $-i$. Once again, one can ask more generally if the Fourier coefficients $a(n)$ with $(n,N)=1$, that belong to the same line, are equidistributed geometrically as above.  

\section{Notions of Density}

Recall that the set of primes (resp. the set of natural numbers) $S\subseteq \mathbb{P}$ (resp. $A\subseteq \mathbb{N}$) has a natural density $d(S)$ (resp. $d(A)$) if the limit $d(S)=\lim_{x\rightarrow +\infty}\frac{\pi_{S}(x)}{\pi(x)}$ (resp. $d(A)=\lim_{x\rightarrow +\infty}\frac{\#\{n\leq x\mid n\in A\}}{x}$) exists, where $\pi_{S}(x)$ and $\pi(x)$ are defined by
$$
\pi(x)=\#\{p\leq x\mid p\in\mathbb{P}\}\text{ and }\pi_{S}(x)=\#\{p\leq x\mid p\in S\}.
$$
The set of primes (resp. of natural numbers) $S$ (resp. $A$) is said to have Dirichlet density $\delta(S)$ (resp. Dedekind-Dirichlet density $\delta(A)$) if the limit \\
$\delta(S)=\lim_{z\rightarrow 1^{+}}\frac{\sum_{p\in S}\frac{1}{p^{z}}}{\text{log }\left(\frac{1}{z-1}\right)}$ (resp. $\delta(A)=\lim_{z\rightarrow 1^{+}}(z-1)\sum_{n\in A}\frac{1}{n^{z}}$) exists. Recall that if the set $A$ of natural numbers has natural density $d(A)$, then it also has Dedekind-Dirichlet density $\delta (A)$ with $d(A)=\delta(A)$. Further, the set of primes $S$ is said to be regular if there is a holomorphic function $g(z)$ on $Re(z)\geq 1$ such that
$$
\sum_{p\in S}\frac{1}{p^{z}}=\delta(S)\text{ log}\left(\frac{1}{z-1}\right)+g(z).
$$
We need the following technical lemma (see \cite[Lemma 2.1]{ilker3}).

\begin{lem}\thlabel{thmdens}
Let $S_1$ and $S_2$ be two regular sets of primes such that $\delta(S_{1})=\delta(S_{2})$. Then the function $\sum_{p\in S_1}\frac{1}{p^{z}}-\sum_{q\in S_1}\frac{1}{q^{z}}$ is analytic on $Re(z)\geq 1$.
\end{lem}

The following proposition said that the set of primes $S$ is regular if it has a natural density that satisfies certain error term (see \cite[Proposition 2.2]{ilker3}). 

\begin{prop}\thlabel{prop1}
Let $S\subseteq \mathbb{P}$ be a set of primes that have natural density $d(S)$. Define $E(x)=\frac{\pi_{S}(x)}{\pi(x)}-d(S)$ to be the error function. Suppose that there are $\alpha>0$, $C>0$, and $M>0$ such that for all $x>M$ we have $\mid E(x)\mid \leq C x^{-\alpha}$. Then $S$ is a regular set of primes.  
\end{prop}

\section{The Chebotarev-Sato-Tate equidistribution}

We recall now some properties of the Shimura lift (see \cite{shimura73}). The Fourier coefficients of $f$ and $F_t$ are related by the following formula

\begin{equation}
A_t(n)=\sum_{d|n}\chi_{t,N}(d)d^{k-1}a\left(\frac{n^2}{d^2}t\right),\label{eq:001}
\end{equation}
 
where $\chi_{t,N}$ denotes the character $\chi_{t,N}(d):=\chi(d)\left(\frac{(-1)^{k}N^{2}t}{d}\right)$. Since $f$ is the Hecke eigenform for the Hecke operator $T_{p^{2}}$, $F_t$ is an eigenform for the Hecke operator $T_p$, for all primes $p\nmid N$. Further, we have $F_{t}=a(t)F$, where $F$ is a normalised Hecke eigenform independant of $t$. 

Applying the Ramanujan-Petersson bound to the Fourier coefficients of $F_t$, then $\mid \frac{A_{t}(p)}{a(t)}\mid \leq 2\,p^{\frac{k-1}{2}}$. Since $F_{t}\in S_{2k}(N/2,\chi^{2})$, then $A_{t}(p)=\chi^{2}(p)\overline{A_{t}(p)}$. Therefore $\frac{A_{t}(p)}{\chi(p)}\in\R$ and define
$$
B_{t}(p):=\frac{A_{t}(p)}{2 a(t)p^{\frac{k-1}{2}}\chi(p)}\in[-1,1].
$$
Notice that $a(t)\in\R$, since $a(t)=\frac{A_{t}(1)}{\chi(1)}$.

Recall that the Sato-Tate measure $\mu$ is the measure on $[-1,1]$ given by $\frac{2}{\pi}\sqrt{1-t^{2}}\,dt$. We state the important Sato-Tate equidistribution theorem for $\Gamma_{0}(N)$ (see Theorem B of \cite{taylor}). 

\begin{thm}\thlabel{thm3}(Barnet-Lamb, Geraghty, Harris, Taylor).
Let $k\geq 1$ and let $F_{t}=\sum_{n\geq 1}A(n)q^{n}\in S_{2k}(N/2,\chi^{2})$ be a cuspidal Hecke eigenform of weight $2k$ for $\Gamma_{0}(N)$. Suppose that $F_{t}$ is without multiplication. Denote by $Im\chi$ the image of $\chi$ and let $\xi\in Im\chi$. Then, when $p$ runs through the primes $p\nmid N$ such that $\chi(p)=\xi$, the numbers $B(p)=\frac{A_{t}(p)}{2 a(t)p^{\frac{k-1}{2}}\chi(p)}\in[-1,1]$ are $\mu-$equidistributed in $[-1,1]$. 
\end{thm}

Inam et al. (see \cite{ilker3},\cite{ilker2},\cite{ilker1}) obtained the equidistribution of the coefficients $a(t n^{2})$ by using \thref{thm3}. In order to prove the geometric equidistribution on the plan as it was explained in the introduction, we need the following hybrid Chebotarev-Sato-Tate equidistribution proved for elliptic curves in \cite{murty} for the first time, and it has been generalized recently by Wong (see \cite{peng}) particularly to non-CM Hecke eigenforms. 

\begin{prop}(Wong)\thlabel{cor2}
Let $q$ be a natural number and $d$ an integer with $(d,q)=1$. Let $[a,b]\subset[-1,1]$ and put $S_{[a,b]}:=\{p\text{ prime}\mid p\equiv d\pmod q, B_{t}(p)\in [a,b]\}$. The set $S_{[a,b]}$ has natural density equal to $\frac{2}{\pi\varphi (q)}\int_{a}^{b}\sqrt{1-t^{2}}\,dt$.  
\end{prop}

Using Dirichlet's theorem on arithmetic progressions, this proposition could be rewritten as follows.

\begin{prop}\thlabel{cor1}
Let $q$ be a natural number, $\varepsilon\pmod q$ a Dirichlet character and $\xi$ a root of unity such that $\xi\in\,Im\,\varepsilon $. Let $[a,b]\subset[-1,1]$ and put $S_{[a,b]}:=\{p\text{ prime}\mid \varepsilon(p)=\xi , B_{t}(p)\in [a,b]\}$. The set $S_{[a,b]}$ has natural density equal to $\frac{1}{\# Im\,\varepsilon}\frac{2}{\pi}\int_{a}^{b}\sqrt{1-t^{2}}\,dt$, where $\# Im\,\varepsilon$ is the cardinality of the image of $\varepsilon $.  
\end{prop}

We will use frequently throughout the paper the following lemma (see \cite{mezroui}).

\begin{lem}
Under the hypothesis fixed in the introduction, let $n$ be an integer such that $(n,N)=1$. Then $\frac{a(t n^{2})}{\chi(n)}\in\R$.
\end{lem}

\section{Preliminaries Results}

We next show that the Chebotarev-Sato-Tate theorem (see \cite[Proposition 2.2]{peng}) gives the equidistribution of the coefficients $a(t p^{2})$ when a primes $p$ run over arithmetic progressions.

\begin{thm}\thlabel{thm4}
We use the assumptions fixed in the introduction and suppose that $F_{t}$ has no CM. Let $q$ be a natural number, $\varepsilon\pmod q$ a Dirichlet character and $\xi$ a root of unity such that $\xi\in\,Im\,\varepsilon $. Define the set of primes
$$
\mathbb{P}_{\varepsilon ,\xi,>}:=\left\{p\in\mathbb{P}\mid \varepsilon(p)=\xi, \frac{a(t p^{2})}{\chi(p)}>0\right\},
$$
and similarly $\mathbb{P}_{\varepsilon,\xi}$, $\mathbb{P}_{\varepsilon,\xi,<}$, $\mathbb{P}_{\varepsilon,\xi ,\geq}$, $\mathbb{P}_{\varepsilon,\xi ,\leq}$, and $\mathbb{P}_{\varepsilon,\xi ,=0}$. Let $d$ be an integer such that $(d,q)=1$. Define also
$$
\mathbb{P}_{d,q,>}:=\left\{p\in\mathbb{P}\mid p\equiv d\text{ mod }q,\frac{a(t p^{2})}{\chi(p)}>0\right\},
$$
and similarly $\mathbb{P}_{d,q}$, $\mathbb{P}_{d,q,<}$, $\mathbb{P}_{d,q,\geq}$, $\mathbb{P}_{d,q,\leq}$, $\mathbb{P}_{d,q,=0}$. 

The sets $\mathbb{P}_{d,q,>}$, $\mathbb{P}_{d,q,<}$, $\mathbb{P}_{d,q,\geq}$, $\mathbb{P}_{d,q,\leq}$ have natural density $\frac{1}{2\varphi(q)}$ and $\mathbb{P}_{d,q,=0}$ has natural density $0$. Further, the sets $\mathbb{P}_{\varepsilon,\xi,>}$, $\mathbb{P}_{\varepsilon,\xi ,<}$, $\mathbb{P}_{\varepsilon,\xi ,\geq}$, $\mathbb{P}_{\varepsilon,\xi ,\leq}$ have natural density $\frac{1}{2\# Im\,\varepsilon}$ and $\mathbb{P}_{\varepsilon,\xi ,=0}$ has natural density $0$, where $\# Im\,\varepsilon$ is the cardinality of the image of $\varepsilon$. 
\end{thm}

\begin{proof}
Define the sets $\pi_{d,q,>}(x):=\#\{p\leq x\mid p\equiv d\text{ mod }q,\frac{a(t p^{2})}{\chi(p)}>0\}$, and similarly, $\pi_{d,q}(x)$,$\pi_{d,q,<}(x)$, $\pi_{d,q,\geq }(x)$, $\pi_{d,q,\leq }(x)$, and $\pi_{d,q,=0}(x)$. Without loss of generality, we can assume that $F_t$ is normalised and thus $a(t)=1$. Denote the character $\left(\frac{(-1)^{k}N^{2}t}{.}\right)$ by $\chi_{1}(.)=\left(\frac{(-1)^{k}N^{2}t}{.}\right)$. The formula \eqref{eq:001} yields
$$
\frac{a(t p^{2})}{\chi(p)}>0 \Longleftrightarrow B_{t}(p)>\frac{\chi_{1}(p)}{2\sqrt{p}}.
$$
Let $\epsilon >0$. Since for all $p>\frac{1}{4 \epsilon^{2}}$, we have $\frac{\chi_{1}(p)}{2\sqrt{p}}=\frac{1}{2\sqrt{p}}<\epsilon$, then
\begin{multline}
\pi_{d,q,>}(x)+\#\{p\leq x \text{ prime}\mid p\equiv d\text{ mod }q, p\leq \frac{1}{4 \epsilon^{2}}\}\geq\\ \# \{p\leq x \text{ prime}\mid p\equiv d\text{ mod }q, B_{t}(p)>\epsilon\}.
\end{multline}
Applying \thref{cor2} we get 
$$
\lim_{x\rightarrow \infty}\frac{\#\{p\leq x\text{ prime}\mid p\equiv d\text{ mod }q, B_{t}(p)>\epsilon\}}{\pi(x)}=\frac{\mu([\epsilon,1])}{\varphi(q)}
$$ 
and then 
$$
\lim_{x\rightarrow \infty}\frac{\#\{p\leq x\text{ prime}\mid p\equiv d\text{ mod }q, B_{t}(p)>\epsilon\}}{\pi_{d,q}(x)}=\mu([\epsilon,1]).
$$
It follows that $\liminf_{x\rightarrow \infty} \frac{\pi_{d,q,>}(x)}{\pi_{d,q}(x)}\geq \mu([\epsilon,1])$ for all $\epsilon >0$, hence $\liminf_{x\rightarrow \infty} \frac{\pi_{d,q,>}(x)}{\pi_{d,q}(x)}\geq \mu([0,1])=\frac{1}{2}$.   

Similarly, we have 
$$
\liminf_{x\rightarrow \infty} \frac{\pi_{d,q,\leq }(x)}{\pi_{d,q}(x)}\geq \mu([0,1])=\frac{1}{2}.
$$ 
Since $\pi_{d,q,\leq }(x)=\pi_{d,q}(x)-\pi_{d,q,>}(x)$, then $\limsup_{x\rightarrow \infty}\frac{\pi_{d,q,>}(x)}{\pi_{d,q}(x)}=\frac{1}{2}$. Using the same method, we obtain the densities of $\mathbb{P}_{d,q,<}$, $\mathbb{P}_{d,q,\geq }$, and $\mathbb{P}_{d,q,\leq }$. Finally, since $\pi_{d,q,= 0}(x)=\pi_{d,q,\geq } (x)-\pi_{d,q,>}(x)$, then the density of $\mathbb{P}_{d,q,=0}(x)$ is zero. 

The densities of the sets $\mathbb{P}_{\varepsilon,\xi,>}$, $\mathbb{P}_{\varepsilon,\xi,<}$, $\mathbb{P}_{\varepsilon,\xi ,\geq}$, $\mathbb{P}_{\varepsilon,\xi ,\leq}$, and $\mathbb{P}_{\varepsilon,\xi ,=0}$ are obtained similarly by using \thref{cor1}. 
\end{proof}

The following theorem said that the set of primes of \thref{thm4} is regular if the Chebotarev-Sato-Tate theorem satisfies certain error term. The proof is closely similar to that of \cite[Theorem 4.2]{ilker3}. 

\begin{thm}\thlabel{th5}
Assuming the assumptions of \thref{thm4} and suppose there are $C>0$ and $\alpha>0$ such that 
$$
\left|\frac{\#\left\{p\leq x\text{ prime}\mid \varepsilon(p)=\xi, \frac{A_{t}(p)}{2 a(t)p^{\frac{k-1}{2}}\chi(p)}\in [a,b]\right\}}{\pi(x)}-\frac{\mu([a,b])}{\# Im\,\varepsilon}\right| \leq \frac{C}{x^{\alpha}}\cdot
$$
Then, the sets $\mathbb{P}_{\varepsilon,\xi, \geq }$, $\mathbb{P}_{\varepsilon,\xi, \leq }$, $\mathbb{P}_{\varepsilon,\xi, >}$, $\mathbb{P}_{\varepsilon,\xi, <}$ and $\mathbb{P}_{\varepsilon,\xi, =0}$ are regular sets of primes. 
\end{thm}

\begin{rmk}
Let $\xi_{q}$ be a qth root of unity. The previous error term is weaker than the one conjectured by Akiyama and Tanigawa 
(see \cite{akiyama}) and it can be obtained by \cite[Theorem 1.3]{peng} if GRH is assumed and also, if $L(z,Sym^{m}\frac{F_{t}}{a(t)}\otimes \eta)$ is automorphic over $\mathbb{Q}$ for every $m$ and for all irreducible characters $\eta$ of $G(\mathbb{Q}(\xi_{q})/\mathbb{Q})$.
\end{rmk}

To proceed with our proof, we establish the following two lemmas.

\begin{lem}\thlabel{thm6}
Assuming the assumptions fixed in the introduction and suppose that $F_{t}$ has no CM. Let $q$ be a natural number. Suppose that for all $\varepsilon\pmod q$ Dirichlet characters and all roots of unity $\xi$ such that $\xi\in\,Im\,\varepsilon$, there are $C_{\varepsilon,\xi}>0$ and $\alpha_{\varepsilon ,\xi}>0$ such that 
\begin{equation}\label{specc}
\left|\frac{\#\left\{p\leq x\text{ prime}\mid p\nmid N, \varepsilon(p)=\xi, \frac{A_{t}(p)}{2 a(t)p^{\frac{k-1}{2}}\chi(p)}\in [a,b]\right\}}{\pi(x)}-\frac{\mu([a,b])}{\# Im\,\varepsilon}\right| \leq \frac{C_{\varepsilon,\xi}}{x^{\alpha_{\varepsilon,\xi}}}. 
\end{equation}
Suppose further that $a(t)>0$. Define the multiplicative function, $\forall n\in\mathbb{N}$,

\begin{displaymath}
f(n)=
\begin{cases}
1, & \text{if $\frac{a(t n^{2})}{\chi(n)}>0$ and $(n,N)=1$,}\\\\
-1, & \text{if $\frac{a(t n^{2})}{\chi(n)}<0$ and $(n,N)=1$,}\\\\
0, & \text{if $a(t n^{2})=0$ and $(n,N)=1$,}\\\\
0, & \text{if $(n,N)\neq 1$.}\\\\
\end{cases}
\end{displaymath}

Let $d$ be an integer with $(d,q)=1$. Then the Dirichlet series 
$$
F(z)=\sum_{\substack{n\geq 1\\n\equiv d\text{ mod }q}}\frac{f(n)}{n^{z}}
$$
is holomorphic on $Re(z)\geq 1$.
\end{lem}

\begin{proof}
We have 

\begin{align*}
\sum_{\substack{n\geq 1\\n\equiv d\text{ mod }q}}\frac{f(n)}{n^{z}}&=\frac{1}{\varphi(q)}\sum_{n=1}^{\infty}\frac{f(n)}{n^{z}}\times \left(\sum_{\varepsilon\text{ mod }q}\varepsilon(n)\overline{\varepsilon (d)}\right)\\
                      &=\frac{1}{\varphi(q)}\sum_{\varepsilon\text{ mod }q}\left(\sum_{n=1}^{\infty}\frac{f(n)\varepsilon (n)}{n^{z}}\right)\times \overline{\varepsilon (d)}.
\end{align*}

Since the first sum is finite, it suffices to show that $G_{\varepsilon}(z)=\sum_{n=1}^{\infty}\frac{f(n)\varepsilon(n)}{n^{z}}$ is holomorphic on $Re(z)\geq 1$. 

Since $a(t)>0$, and $\forall m,n\in\mathbb{N}$, $(m,N)=1$, $(n,N)=1$,
$$
\frac{a(t m^{2})}{\chi(m)}\frac{a(t n^{2})}{\chi(n)}=a(t)\frac{a(t m^{2} n^{2})}{\chi(mn)},
$$
then $f(n)$ is multiplicative. 

Applying \cite[Lemma 2.1.2]{ilker2}, we obtain 
$$
\text{log } G_{\varepsilon}(z)=\sum_{p\in\mathbb{P}}\frac{f(p)\varepsilon(p)}{p^{z}}+g(z),
$$
where $g(z)$ is a function that is holomorphic on $Re(z)>\frac{1}{2}$. Hence
\begin{align*}
\text{log } G_{\varepsilon}(z)&=\sum_{p\in\mathbb{P}}\frac{f(p)\varepsilon(p)}{p^{z}}+g(z)\\
                      &=\sum_{\xi\in\,Im(\varepsilon)}\xi\sum_{p\in\mathbb{P}_{\varepsilon,\xi}}\frac{f(p)}{p^{z}}+g(z)\\
                      &=\sum_{\xi\in\,Im(\varepsilon)}\xi\left(\sum_{p\in\mathbb{P}_{\varepsilon,\xi,>}}\frac{1}{p^{z}}-\sum_{p\in\mathbb{P}_{\varepsilon,\xi,<}}\frac{1}{p^{z}}\right)+g(z).
\end{align*}
The sets $\mathbb{P}_{\varepsilon,\xi,>}$ and $\mathbb{P}_{\varepsilon,\xi,<}$ are regular sets of primes, and they have the same density $\frac{1}{2\# Im\,\varepsilon}$ by \thref{thm4}. Therefore by \thref{thmdens}, $\text{log } G_{\varepsilon}(z)$ is holomorphic on $R(z)\geq 1$, and consequently $G_{\varepsilon}(z)$ is also holomorphic.
\end{proof}

\begin{lem}\thlabel{leme1}
We use the assumptions fixed in the introduction and suppose that $F_{t}$ has no CM. Let $q$ be a natural number. Suppose that for all Dirichlet characters $\varepsilon\pmod q$ and all roots of unity $\xi$ such that $\xi\in\,Im\,\varepsilon$, there are $C_{\varepsilon,\xi}>0$ and $\alpha_{\varepsilon ,\xi}>0$ such that 
\begin{equation}
\left|\frac{\#\left\{p\leq x\text{ prime}\mid p\nmid N, \varepsilon(p)=\xi, \frac{A_{t}(p)}{2 a(t)p^{\frac{k-1}{2}}\chi(p)}\in [a,b]\right\}}{\pi(x)}-\frac{\mu([a,b])}{\# Im\,\varepsilon}\right| \leq \frac{C_{\varepsilon,\xi}}{x^{\alpha_{\varepsilon,\xi}}}. 
\end{equation}
Then for all integers $d$, $(d,q)=1$, the set 
$$
\{n\in\mathbb{N}\mid (n,N)=1, n\equiv d\text{ mod }q, a(t n^{2})\neq 0\}
$$
has natural density.
\end{lem}

\begin{proof}
We have 
$$
\sum_{\substack{n\geq 1\\n\equiv d\text{ mod }q}}\frac{f(n)^{2}}{n^{z}}=\frac{1}{\varphi(q)}\sum_{\varepsilon\text{ mod }q}\left(\sum_{n=1}^{\infty}\frac{f(n)^{2}\varepsilon (n)}{n^{z}}\right)\times \overline{\varepsilon (d)}.
$$
We shall define $H_{\varepsilon}(z)=\sum_{n=1}^{\infty}\frac{f(n)^{2}\varepsilon (n)}{n^{z}}$. Applying \cite[Lemma 2.1.2]{ilker2} to get
\begin{align*}
\text{log } H_{\varepsilon}(z):=&\sum_{p\in\mathbb{P}}\frac{f(p)^{2}\varepsilon(p)}{p^{z}}+g_{\varepsilon}(z)\\
=&\sum_{\xi\in\,Im\,\varepsilon}\xi\left(\sum_{p\in\mathbb{P}_{\varepsilon,\xi,>}\cup \mathbb{P}_{\varepsilon,\xi,<}}\frac{1}{p^{z}}\right)+g_{\varepsilon}(z),
\end{align*}
where $g_{\varepsilon}(z)$ is a function that is holomorphic on $Re(z)>\frac{1}{2}$. Applying \thref{th5}, the sets $\mathbb{P}_{\varepsilon,\xi,>}$ and $\mathbb{P}_{\varepsilon,\xi,<}$ are regular sets of primes of natural density $\frac{1}{2\#\,Im\,\varepsilon}$. Then 
$$
\sum_{p\in\mathbb{P}_{\varepsilon,\xi,>}\cup \mathbb{P}_{\varepsilon,\xi,<}}\frac{1}{p^{z}}=\frac{1}{\#\,Im\,\varepsilon}\text{log }\left(\frac{1}{z-1}\right)+h_{\xi}(z),
$$
where $h_{\xi}$ is a holomorphic function on $Re(z)\geq 1$. It follows that 
\begin{align*}
\text{log } H_{\varepsilon}(z):=&\sum_{\xi\in\,Im\,\varepsilon}\xi\left(\sum_{p\in\mathbb{P}_{\varepsilon,\xi,>}\cup \mathbb{P}_{\varepsilon,\xi,<}}\frac{1}{p^{z}}\right)+g_{\varepsilon}(z)\\
=&\frac{\sum_{\xi\in\,Im\,\varepsilon}\xi}{\#\,Im\,\varepsilon}\text{log }\left(\frac{1}{z-1}\right)+\sum_{\xi\in\,Im\,\varepsilon}\xi \,h_{\xi}(z)+g_{\varepsilon}(z).
\end{align*}
Thus $\text{log } H_{\varepsilon_{0}}(z)=\text{log }\left(\frac{1}{z-1}\right)+h_{1}(z)+g_{\varepsilon_{0}(z)}$ where $\varepsilon_{0}$ is the principal Dirichlet character modulo $q$, and $\text{log } H_{\varepsilon}(z)=\sum_{\xi\in\,Im\,\varepsilon}\xi \,h_{\xi}(z)+g_{\varepsilon}(z)$ when $\varepsilon\neq \varepsilon_{0}$. From this we see that in all cases, there is $b_{\varepsilon}\in\mathbb{C}$ satisfying
$$
H_{\varepsilon}(z)=\frac{b_{\varepsilon}}{z-1}+k_{\varepsilon}(z),
$$
where $k_{\varepsilon}$ is holomorphic on $Re(z)\geq 1$. Therefore 
$$
\sum_{\substack{n\geq 1\\n\equiv d\text{ mod }q}}\frac{f(n)^{2}}{n^{z}}=\frac{b}{z-1}+k(z),
$$
where $b\in\mathbb{C}$ and $k$ is holomorphic on $Re(z)\geq 1$. We can now apply Wiener-Ikehara's theorem (see\cite{kor}) to deduce the result. 

\end{proof}

\begin{rmk}\thlabel{rem2}
Notice that the natural density of the set
$$
\{n\in\mathbb{N}\mid (n,N)=1, n\equiv d\text{ mod }q, a(t n^{2})\neq 0\}
$$
is independent of the choice of $d$. Indeed, from Wiener-Ikehara's theorem we know that this density is equal to $\frac{h_{1}(1)+g_{\varepsilon_{0}}(1)}{\varphi(q)}$. 
\end{rmk}

\section{Proof of \texorpdfstring{\thref{thm71}}{Theorem 1}}

Before starting the proof, recall the theorem of Delange (see \cite{dela}).
 
\begin{thm}
Let $g:\mathbb{N}\longrightarrow \mathbb{C}$ be a multiplicative arithmetic function for which:
\begin{enumerate}
\item $\forall n\in\mathbb{N},\mid g(n)\mid \leq 1$.
\item There exists $a\in\mathbb{C}$ such that $a\neq 1$ and satisfying $\frac{\lim_{x\rightarrow +\infty}\sum_{\substack{p\text{ prime}\\p\leq x}}g(p)}{\pi(x)}=a$. 
\end{enumerate}
Then we have
$$
\frac{\lim_{x\rightarrow +\infty}\sum_{n\leq x}g(n)}{x}=0.
$$
\end{thm}

We can now piece together the previous lemmas to prove \thref{thm71}.

\begin{proof}

We have
\begin{equation}
\sum_{\substack{1\leq n\leq x\\n\equiv d\text{ mod }q}}f(n)=\frac{1}{\varphi(q)}\sum_{\varepsilon\text{ mod }q}\left(\sum_{1\leq n\leq x}f(n)\varepsilon (n)\right)\times \overline{\varepsilon (d)}.
\end{equation}
For a Dirichlet character $\varepsilon$ modulo $q$, we have
\begin{align*}
\lim_{x\rightarrow +\infty}\frac{\sum_{1\leq p\leq x}f(p)\varepsilon (p)}{\pi(x)}=&\lim_{x\rightarrow +\infty}\sum_{\xi\in\,Im\,\varepsilon}\xi\frac{\#\{p\leq x\mid p\in\mathbb{P}_{\varepsilon,\xi,>}\}}{\pi(x)}-\\
&\xi\frac{\#\{p\leq x\mid p\in\mathbb{P}_{\varepsilon,\xi,<}\}}{\pi(x)}\\
=&0,
\end{align*}
since $\mathbb{P}_{\varepsilon,\xi,>}$ and $\mathbb{P}_{\varepsilon,\xi,<}$ have the same natural density $\frac{1}{2\#\,Im\,\varepsilon }$. Applying Delange's theorem, we get $\lim_{x\rightarrow +\infty}\frac{\sum_{1\leq n\leq x}f(n)\varepsilon (n)}{x}=0$, and consequently
$$
\lim_{x\rightarrow +\infty}\frac{\sum_{\substack{1\leq n\leq x\\n\equiv d\text{ mod }q}}f(n)}{x}=0.
$$
From which we have  

\begin{multline}\label{eq:fin1}
\lim_{x\rightarrow +\infty}\frac{\#\left\{n\leq x\mid (n,N)=1, n\equiv d\text{ mod }q, \frac{a(t n^{2})}{\chi(n)}>0\right\}}{x}-\\
\frac{\#\left\{n\leq x\mid (n,N)=1, n\equiv d\text{ mod }q,\frac{a(t n^{2})}{\chi(n)}<0\right\}}{x}=0.
\end{multline}

By \thref{leme1}, there is $b>0$ such that  

\begin{multline}\label{eq:fin2}
\lim_{x\rightarrow +\infty}\frac{\#\left\{n\leq x\mid (n,N)=1, n\equiv d\text{ mod }q, \frac{a(t n^{2})}{\chi(n)}>0\right\}}{x}+\\
\frac{\#\left\{n\leq x\mid (n,N)=1, n\equiv d\text{ mod }q,\frac{a(t n^{2})}{\chi(n)}<0\right\}}{x}=b.
\end{multline}

The result follows from \eqref{eq:fin1} and \eqref{eq:fin2}.

\end{proof}

We show finally by another method how the natural density of the set defined in \thref{leme1} is independent of $d$.  

\begin{prop}\thlabel{rem1}
Assuming the assumptions of the main theorem. Then, the natural density of the set
$$
\{n\in\mathbb{N}\mid (n,N)=1, n\equiv d\text{ mod }q, a(t n^{2})\neq 0\}
$$
is equal to 
$$
\frac{1}{\varphi(q)}\lim_{z\rightarrow 1^{+}}(z-1)\sum_{\substack{n=1\\(n,q)=1}}^{\infty}\frac{f(n)^2}{n^{z}}.
$$
\end{prop}

\begin{proof}
Since $\{n\in\mathbb{N}\mid (n,N)=1, n\equiv d\text{ mod }q, a(t n^{2})\neq 0\}$ has natural density by \thref{leme1}, then it suffices to prove that the Dedekind-Dirichlet density of this set is equal to $\frac{1}{\varphi(q)}\lim_{z\rightarrow 1^{+}}(z-1)\sum_{\substack{n=1\\(n,q)=1}}^{\infty}\frac{f(n)^2}{n^{z}}$.

We shall define $B(z)=\sum_{\substack{n=1\\n\equiv d\text{ mod }q}}^{\infty}\frac{f(n)^2}{n^{z}}$ and $C_{\varepsilon }(z)=\sum_{n=1}^{\infty}\frac{f(n)^{2}\varepsilon(n)}{n^{z}}$ where $\varepsilon$ runs over Dirichlet characters modulo $q$. We must now compute $\lim_{z\rightarrow 1^{+}}(z-1)B(z)$. By the same computations as in the previous theorem, il suffices to compute $\lim_{z\rightarrow 1^{+}}(z-1)C_{\varepsilon}(z)$. We have

\begin{align*}
\frac{C_{\varepsilon}(z)}{L(z,\varepsilon)}&=\prod_{\substack{p\in\mathbb{P}}}\sum_{k=0}^{\infty}f(p^{k})^2\varepsilon(p^{k}) p^{-kz}\times\prod_{p\in\mathbb{P}}(1-\frac{\varepsilon(p)}{p^{z}})\\
                      &=\prod_{\substack{p\in\mathbb{P}}}(1-\frac{\varepsilon(p)}{p^{z}})\times\prod_{\substack{p\in\mathbb{P}}}\left(1+\sum_{\substack{k=1\\a(t p^{2k})\neq 0}}^{\infty}\frac{\varepsilon(p^{k})}{p^{kz}}\right)\\
                      &=\prod_{\substack{p\in\mathbb{P}\\a(t p^{2})\neq 0}}\left[(1-\frac{\varepsilon(p)}{p^{z}})\left(1+\frac{\varepsilon(p)}{p^{z}}+\sum_{\substack{k=2\\a(t p^{2k})\neq 0}}^{\infty}\frac{\varepsilon(p^{k})}{p^{kz}}\right)\right]\\
                      &\times \prod_{\substack{p\in\mathbb{P}\\a(t p^{2})=0}}\left[(1-\frac{\varepsilon(p)}{p^{z}})\left(1+\sum_{\substack{k=2\\a(t p^{2k})\neq 0}}^{\infty}\frac{\varepsilon(p^{k})}{p^{kz}}\right)\right]\\
                      &=\prod_{\substack{p\in\mathbb{P}\\a(t p^{2})\neq 0}}\left(1-\frac{\varepsilon(p^{2})}{p^{2z}}+h_{1}(z,p)\right)\times\prod_{\substack{p\in\mathbb{P}\\a(t p^{2})=0}}\left(1-\frac{\varepsilon(p)}{p^{z}}+h_{2}(z,p)\right),
\end{align*}

where $h_{1}(z,p)$ and $h_{2}(z,p)$ are the remaining terms. Applying logarithm to $\frac{C_{\varepsilon}(z)}{L(z,\varepsilon)}$ and notice that $\sum_{\substack{p\in\mathbb{P}\\a(t p^{2})\neq 0}}\text{log } \left(1-\frac{\varepsilon(p^{2})}{p^{2z}}+h_{1}(z,p)\right)$ is holomorphic on $Re(z)\geq 1$. On the other hand, we have $\sum_{\substack{p\in\mathbb{P}\\a(t p^{2})=0}}\text{log }\left(1-\frac{\varepsilon(p)}{p^{z}}+h_{2}(z,p)\right)=\sum_{\substack{p\in\mathbb{P}\\a(t p^{2})=0}}\frac{\varepsilon(p)}{p^{z}}+h_{3}(z,p)$ where $h_{3}(z,p)$ is holomorphic on $Re(z)\geq 1$. Further, since for all roots of unity $\xi$ such that $\xi\in\,Im\,\varepsilon$, the set $\mathbb{P}_{\varepsilon ,\xi,=0}$ is a regular set of primes of density $0$ by \thref{thm4}, then 
$$
\sum_{\substack{p\in\mathbb{P}\\a(t p^{2})=0}}\frac{\varepsilon (p)}{p^{z}}=\sum_{\xi\in\,Im\,\varepsilon }\xi\sum_{\substack{p\in\mathbb{P}_{\varepsilon,\xi,=0}}}\frac{1}{p^{z}}
$$
is also holomorphic on $Re(z)\geq 1$. Thus $\text{log } \frac{C_{\varepsilon}(z)}{L(z,\varepsilon)}$ is holomorphic on $Re(z)\geq 1$ and by taking exponential we see that $\frac{C_{\varepsilon}(z)}{L(z,\varepsilon)}$ is also holomorphic on $Re(z)\geq 1$. Then the limit $\lim_{z\rightarrow 1^{+}}(z-1)C_{\varepsilon_{0}}(z)$ exists, where $\varepsilon_{0}$ is the principal character modulo $q$, and $\lim_{z\rightarrow 1^{+}}(z-1)C_{\varepsilon}(z)=0$ when $\varepsilon\neq\varepsilon_{0}$. 

\begin{align*}
\lim_{z\rightarrow 1^{+}}(z-1)B(z)=&\frac{1}{\varphi(q)}\lim_{z\rightarrow 1^{+}}(z-1)C_{\varepsilon_{0}}(z)\\
=&\frac{1}{\varphi(q)}\lim_{z\rightarrow 1^{+}}(z-1)\sum_{\substack{n=1\\(n,q)=1}}^{\infty}\frac{f(n)^2}{n^{z}}.
\end{align*}

\end{proof}

We conclude with some related remarks.

\begin{rmk}
When $q=N$ or $(q,N)=1$, the Dedekind-Dirichlet density of the set $\{n\in\mathbb{N}\mid (n,N)=1, n\equiv d\text{ mod }q, a(t n^{2})=0\}$ exists . Indeed, we have 
$$
\lim_{z\rightarrow 1^{+}}(z-1)\sum_{\substack{n\geq 1\\ n\equiv d\text{ mod }q}}\frac{1}{n^{z}}=\frac{1}{q}.
$$
By \thref{thm6}, it follows that

\begin{equation}\label{eq:11s}
\lim_{z\rightarrow 1^{+}}(z-1)\left(2 \sum_{\substack{(n,N)=1\\\frac{a(t n^{2})}{\chi(n)}>0\\n\equiv d\text{ mod }q}}\frac{1}{n^{z}}+\sum_{\substack{(n,N)=1\\a(t n^{2})=0\\ n\equiv d\text{ mod }q}}\frac{1}{n^{z}}+\sum_{\substack{(n,N)\neq 1\\ n\equiv d\text{ mod }q}}\frac{1}{n^{z}}\right)=\frac{1}{q}\cdot
\end{equation}

Let $\chi_{0}$ be a principal character modulo $N$. We have 

\begin{align*}
\sum_{\substack{(n,N)=1\\ n\equiv d\text{ mod }q}}\frac{1}{n^{z}}=&\sum_{n\equiv d\text{ mod }q}\frac{\chi_{0}(n)}{n^{z}}\\
=&\frac{1}{\varphi(q)}\sum_{n\geq 0}\frac{\chi_{0}(n)}{n^{z}}\sum_{\varepsilon\text{ mod }q}\overline{\varepsilon(d)}\varepsilon(n)\\
=&\frac{1}{\varphi(q)}\sum_{\varepsilon\text{ mod }q}\overline{\varepsilon(d)}\sum_{n\geq 0}\frac{\chi_{0}(n)\varepsilon(n)}{n^{z}}.
\end{align*}

Following our hypothesis, if $q=N$ we consider $\chi_{0}\varepsilon$ as a character modulo $N$, if $(q,N)=1$ we consider it as a character modulo $q N$. Therefore $\lim_{z\rightarrow 1+}\sum_{\substack{(n,N)=1\\ n\equiv d\text{ mod }q}}\frac{1}{n^{z}}$ exists and thus $\lim_{z\rightarrow 1+}\sum_{\substack{(n,N)\neq 1\\ n\equiv d\text{ mod }q}}\frac{1}{n^{z}}$ also exists. Replacing this in \eqref{eq:11s} and the result follows. 
\end{rmk}

\begin{rmk}
The weaker version of \thref{thm71} could be obtained using \thref{rem1}. Indeed, in the proof of the previous proposition there is $b>0$ such that $\lim_{z\rightarrow 1^{+}}(z-1)B(z)=b$. Hence $\{n\in\mathbb{N}\mid (n,N)=1, n\equiv d\text{ mod }q\text{ and }a(t n^{2})\neq 0\}$ has a Dedekind-Dirichlet density equal to $b$. It follows from \eqref{eq:11s} that 
$$
\lim_{z\rightarrow 1^{+}}(z-1)\left(\sum_{\substack{(n,N)=1\\n\equiv d\text{ mod }q\\a(t n^{2})=0}}\frac{1}{n^{z}}+\sum_{\substack{(n,N)\neq 1\\ n\equiv d\text{ mod }q}}\frac{1}{n^{z}}\right)=\frac{1}{q}-b.
$$
Replace this in \eqref{eq:11s} to get
$$
\lim_{z\rightarrow 1^{+}}(z-1)\sum_{\substack{(n,N)=1\\n\equiv d\text{ mod }q\\\frac{a(t n^{2})}{\chi(n)}>0}}\frac{1}{n^{z}}=\frac{b}{2}
$$
The equidistribution obtained here is in terms of the Dedekind-Dirichlet density only.
\end{rmk}

\bibliographystyle{spmpsci}
\bibliography{mezsoufiane}

\end{document}